\documentclass[draft,a4paper,11pt]{amsart}

\usepackage{amssymb}
\usepackage{color}
\usepackage{graphicx}
\usepackage{bbm}

\newtheorem{theorem}{Theorem}[section]
\newtheorem{lemma}[theorem]{Lemma}
\newtheorem{prop}[theorem]{Proposition}

\newtheorem{cor}[theorem]{Corollary}

\theoremstyle{definition}
\newtheorem{definition}[theorem]{Definition}
\newtheorem{example}[theorem]{Example}

\numberwithin{theorem}{section}

\newcommand{\LQ}{\mathbbm{Q}}
\newcommand{\LR}{\mathbbm{R}}
\newcommand{\LZ}{\mathbbm{Z}}

\newdimen{\standardlabelwidth}
\standardlabelwidth-\standardlabelwidth
  \advance\standardlabelwidth by 2\normalparindent
\newcommand{\standardlabel}[1]{#1\kern\standardlabelwidth}

\begin{document}


\title[A note on coincidence isometries of modules in
  Euclidean space]{A note on coincidence isometries\\ of modules in
  Euclidean space}
\author{Christian Huck}
\address{Department of Mathematics and Statistics\\
  The Open University\\ Walton Hall\\ Milton Keynes\\ MK7 6AA\\
   United Kingdom}
\email{c.huck@open.ac.uk}
\thanks{The author was supported by EPSRC via Grant EP/D058465/1.}

\begin{abstract}
It is shown that the coincidence isometries of certain modules in Euclidean
$n$-space can be decomposed into a product of at most $n$ coincidence reflections
  defined by their non-zero elements. This generalizes previous results obtained for lattices to situations that are relevant in quasicrystallography.
\end{abstract}

\maketitle

\section{Introduction}\label{sec1}

The classification of grain boundaries in (quasi)crystals
is closely related to the existence of coincidence submodules of their 
underlying $\LZ$-modules; cf.~\cite{baake}, ~\cite{bpr},
~\cite{grimmer}, ~\cite{PBR} and references therein. Here, two $\LZ$-modules (additive subgroups) $\varGamma,\varGamma'$ of Euclidean $n$-space are called
{\em commensurate}, denoted by $\varGamma\sim \varGamma'$, if their
intersection $\varGamma\cap \varGamma'$ has finite subgroup index both in
$\varGamma$ and $\varGamma'$. A (linear) {\em coincidence isometry} of
$\varGamma$ is an element $f$ of the real orthogonal group $\operatorname{O}(n)$ such that
$\varGamma\sim f(\varGamma)$. We denote by $\operatorname{OC}(\varGamma)$ the set of linear coincidence 
isometries of $\varGamma$. For $f\in\operatorname{OC}(\varGamma)$,  the intersection $\varGamma\cap
f(\varGamma)$ is called a {\em coincidence submodule} of
$\varGamma$. In fact, for finitely generated $\LZ$-modules $\varGamma$, the set $\operatorname{OC}(\varGamma)$ is a subgroup
of $\operatorname{O}(n)$ 
(Corollary~\ref{group}). For the main results, we
consider free $S$-modules
$\varGamma\subset\LR^n$ of rank $n$ that span $\LR^n$, where $S$ is a
subring of $\LR$ that is finitely generated as a $\LZ$-module, and extend
previous results obtained
for the crystallographic case $S=\LZ$, where $\varGamma$ is a lattice,
to the more general situation;
cf.~\cite{zou}. In particular,
we consider {\em $S$-modules over $K$ in $\LR^n$} 
(Definition~\ref{omodule}), where $K$ is the field of fractions of
$S$, and show that the coincidence
isometries of these modules $\varGamma$ can be decomposed into at
most $n$ coincidence reflections defined by non-zero elements of $\varGamma$
(Corollary~\ref{ci}). This class of modules includes many examples
that are
relevant in (quasi)crystallography (Example~\ref{ex}).

\section{Preliminaries and general results}\label{sec3}

Frequently, and without further mention, we use the consequence of
Lagrange's theorem saying that, for any subgroup $\varGamma'$ of an
Abelian group $\varGamma$ of finite subgroup index
$m:=[\varGamma:\varGamma']$, the group $m\varGamma$ is a subgroup of $\varGamma'$; compare~\cite[Ch.~1, Propositions 2.2
  and 4.1]{La}.

\begin{lemma}\label{comm0}
Let $\varGamma,\varGamma'\subset\LR^n$ be free $\LZ$-modules of finite
rank $r$. The following assertions are equivalent:
\begin{itemize}
\item[{\rm (i)}]
$\varGamma\sim \varGamma'$.
\item[{\rm (ii)}]
$\varGamma\cap\varGamma'$ contains a free $\LZ$-module of rank $r$.
\item[{\rm (iii)}]
$\varGamma\cap\varGamma'$ is a free $\LZ$-module of rank $r$.
\end{itemize}
\end{lemma}
\begin{proof}
If $\varGamma\sim \varGamma'$, then the subgroup index 
$m=[\varGamma:(\varGamma\cap \varGamma')]$ is finite. Hence
$\varGamma\cap \varGamma'$ contains the free $\LZ$-module $m\varGamma$
of rank $r$. This proves direction (i) $\Rightarrow$ (ii). Since
$\varGamma$ is a free $\LZ$-module of rank $r$ and
since $\LZ$ is a principal ideal
domain, direction (ii) $\Rightarrow$ (iii) follows from~\cite[Ch.~3, Theorem 7.1]{La}. Direction
(iii) $\Rightarrow$ (i) is an immediate consequence of~\cite[Ch.~2, Lemma 6.1.1]{bo}.
\end{proof}
 
\begin{lemma}\label{comm}
Commensurateness of free $\LZ$-modules $\varGamma \subset\LR^n$ of the
same 
finite rank is
an equivalence relation.
\end{lemma}
\begin{proof}
Reflexivity and symmetry are clear by definition. For the
transitivity, let $\varGamma_1\sim \varGamma_2$ and $\varGamma_2\sim \varGamma_3$. In particular, the
indices $m_{12}=[\varGamma_2:(\varGamma_1\cap \varGamma_2)]$ and $m_{23}=[\varGamma_2:(\varGamma_2\cap \varGamma_3)]$ are
finite. Since $m_{12}\varGamma_2\subset (\varGamma_1\cap \varGamma_2)$ and
$m_{23}\varGamma_2\subset (\varGamma_2\cap \varGamma_3)$, one sees
that $\varGamma_1\cap \varGamma_3$ contains the free $\LZ$-module  
$m_{12}m_{23}\varGamma_2$, which has the same finite rank as $\varGamma_1$ and $\varGamma_3$, whence $\varGamma_1\sim
\varGamma_3$ by Lemma~\ref{comm0}. 
\end{proof}

\begin{cor}\label{group}
For free $\LZ$-modules $\varGamma\subset\LR^n$ of finite rank,
$\operatorname{OC}(\varGamma)$ is a group. In particular, it is a
subgroup of $\operatorname{O}(n)$.
\end{cor}
\begin{proof}
Since commensurateness of free $\LZ$-modules 
of the same finite rank is reflexive (Lemma~\ref{comm}), one has $\operatorname{id}\in\operatorname{OC}(\varGamma)$. Now let
$f,g\in\operatorname{OC}(\varGamma)$, i.e., $\varGamma\sim f(\varGamma)$ and $\varGamma\sim
g(\varGamma)$. By the symmetry of commensurateness of free
$\LZ$-modules of the same finite rank (Lemma~\ref{comm}), the latter implies $\varGamma\sim g^{-1}(\varGamma)$ and, further, that
$f(\varGamma)\sim
fg^{-1}(\varGamma)$. Finally, the transitivity
property of commensurateness of free $\LZ$-modules of the same finite
rank gives $\varGamma\sim fg^{-1}(\varGamma)$ (Lemma~\ref{comm}).
\end{proof}

We use the convention that subrings of $\LR$ always contain the
number $1$. Recall that an algebraic number field $K$ is a field extension 
of $\LQ$ of finite degree $d:=[K:\LQ]$ over $\LQ$. An order $\mathcal{O}$ of $K$ is
a subring of $K$ that is free 
 of rank $d$ as a $\LZ$-module. Among the various orders of $K$, there
is one maximal order which contains all the other orders, namely the ring $\mathcal{O}_K$
 of algebraic integers in $K$; cf.~\cite{bo} for general background
 material on algebraic number fields. As a generalization of the ring
 $\LZ$ of 
 rational integers, we now work with subrings $S$ of $\LR$ that are finitely
generated as a $\LZ$-module and hence are free $\LZ$-modules of finite
rank $s$ by~\cite[Ch.~3, Theorem 7.3]{La}. These rings can be characterized as follows.

\begin{prop}
Let $S$ be a subring of\/ $\LR$ that is finitely
generated as a $\LZ$-module. Then, its field $K$ of fractions in $\LR$
is a real algebraic number field and $S$ is contained in $\mathcal{O}_K$, with
equality if and only if $S$ is a Dedekind domain. Vice
versa, every subring $S$ of the ring of algebraic integers in a real algebraic
number field $K$ is finitely
generated as a $\LZ$-module.
\end{prop}
\begin{proof}
Let $S$ be a subring of\/ $\LR$ that is finitely
generated as a $\LZ$-module, say $S=\LZ\alpha_1+ \dots
+\LZ\alpha_m$. By the explanations given in \S1 of~\cite[Ch.~7]{La}, every element of $S$
is integral over $\LZ$, i.e., $S\subset\mathcal{O}_K$. Since
$S=\LZ[\alpha_1,\dots,\alpha_m]$, one has
$$K=\LQ(\alpha_1,\dots,\alpha_m)\,,$$ whence $K$ is indeed a real
algebraic number field; cf.~\cite[Ch.~5, Proposition 1.6]{La}. Suppose
that $S$ is also a Dedekind domain and let
$\alpha\in\mathcal{O}_K$. Trivially, this implies that
$\alpha$ is integral over $S$. But this means that $\alpha\in S$
because, in particular, $S$ is
integrally closed. Hence, for Dedekind domains $S$, one has $S=\mathcal{O}_K$. On the other
hand, rings $\mathcal{O}_K$ of algebraic integers in a real
algebraic number field $K$ are always Dedekind domains; cf.~\cite{N}. Now let $S$ be a subring of the
ring of algebraic integers $\mathcal{O}_K$ in a real algebraic
number field $K$. Since $\mathcal{O}_K$ is free of rank
$d:=[K:\LQ]$ as a $\LZ$-module, $S$ is finitely
generated as a $\LZ$-module; cf.~\cite[Ch.~3, Theorem 7.1]{La}.
\end{proof}

Further, we consider free $S$-modules $\varGamma\subset\LR^n$ of rank $n$
that span $\LR^n$. In other words, these modules $\varGamma$ are the $S$-span of an
$\LR$-basis of $\LR^n$. In particular, $\varGamma$ is a free $\LZ$-module of rank
$s n$. In the simplest case $S=\LZ$, $\varGamma$ is a
{\em lattice} in $\LR^n$. Further, $K$ will always denote the field of
 fractions of $S$ in $\LR$. There is the following characterization of commensurateness; compare~\cite[Theorem
  2.1]{zou} for the lattice case. Originally, the proof in the
case of lattices in Euclidean $3$-space is due
to Grimmer; cf.~\cite{grimmer}.

\begin{theorem}\label{charac}
Let
$\varGamma_1,\varGamma_2\subset\LR^n$ be free $S$-modules
of rank $n$ that span 
$\LR^n$ and let $B_1,B_2\in\operatorname{GL}(n,\LR)$
be basis matrices of the $S$-modules $\varGamma_1$ and $\varGamma_2$,
respectively. The following assertions are equivalent:
\begin{itemize}
\item[{\rm (i)}]
$\varGamma_1\sim \varGamma_2$.
\item[{\rm (ii)}]
The intersection $\varGamma_1\cap \varGamma_2$ contains a free
$S$-module of rank $n$ that spans $\LR^n$.
\item[{\rm (iii)}] 
$B_2^{-1}B_1\in\operatorname{GL}(n,K)$.
\end{itemize}
\end{theorem}
\begin{proof}
If $\varGamma_1\sim \varGamma_2$, the subgroup index
$m=[\varGamma:(\varGamma\cap \varGamma')]$ is finite. Hence, 
$\varGamma\cap \varGamma'$ contains the free $S$-module $m\varGamma$
of rank $n$ that spans
$\LR^n$. This proves direction (i) $\Rightarrow$ (ii). For direction
(ii) $\Rightarrow$ (iii), let $B\in\operatorname{GL}(n,\LR)$ be the basis matrix of an $\LR$-basis
contained in $\varGamma\cap \varGamma'$. This implies the existence of non-singular matrices
$Z_1,Z_2\in\operatorname{Mat}(n,S)$ such that
$$
B_1Z_1=B=B_2Z_2\,,
$$
whence $B_2^{-1}B_1=Z_2Z_1^{-1}\in\operatorname{GL}(n,K)$ by the
standard formula for the inverse of a matrix. Finally, for direction
(iii) $\Rightarrow$ (i), assume that 
$B_2^{-1}B_1\in\operatorname{GL}(n,K)$. Then, there is a non-zero number
$\alpha\in S$ such that $B:=\alpha
B_2^{-1}B_1\in\operatorname{Mat}(n,S)$. The identity $\alpha
B_1=B_2B$ now implies $\alpha\varGamma_1\subset
(\varGamma_1\cap\varGamma_2)$. Since $\alpha\varGamma_1$ is a free
$\LZ$-module of rank $s n$, Lemma~\ref{comm0} implies $\varGamma_1\sim \varGamma_2$.
\end{proof}

\begin{cor}\label{characcor}
Let
$\varGamma\subset\LR^n$ be a free $S$-module of rank $n$ that spans
$\LR^n$. Then, for all basis matrices
$B_{\varGamma}\in\operatorname{GL}(n,\LR)$ of the
$S$-module $\varGamma$, one has
$$
\operatorname{OC}(\varGamma)\simeq\operatorname{O}(n,\LR)\cap(B_{\varGamma}\operatorname{GL}(n,K)B_{\varGamma}^{-1})
$$   
via the map that assigns to an element of $\operatorname{O}(n)$ its
representing matrix with respect to the canonical basis of $\LR^n$. 
\end{cor}
\begin{proof}
With 
$R\in\operatorname{O}(n,\LR)$, $RB_{\varGamma}$ is a basis matrix
of the $S$-module $R\varGamma$. The assertion now follows from Theorem~\ref{charac}. 
\end{proof}

\begin{example}
For a free
$S$-module $\varGamma$ of rank $n$ that spans
$\LR^n$ and is contained in $K^n$ (and hence has a basis matrix
$B_{\varGamma}\in\operatorname{GL}(n,K)$), Corollary~\ref{characcor}
shows that
$$\operatorname{OC}(\varGamma)\simeq\operatorname{O}(n,K):=\operatorname{O}(n,\LR)\cap\operatorname{GL}(n,K)\,.$$
In particular, this applies
to the $S$-module $S^n$.  
\end{example}

By definition, a
 {\em similarity isometry} of $\varGamma$ is an element
  $f\in\operatorname{O}(n)$ such that $\varGamma\sim\alpha
  f(\varGamma)$ for a suitable positive real number $\alpha$;
 compare~\cite{bghz}. For a relation between the group $\operatorname{OC}(\varGamma)$ and
  its supergroup $\operatorname{OS}(\varGamma)$ of similarity
 isometries of $\varGamma$, we refer the reader to~\cite{glied}. Our
 aim is now to gain some insight into the structure
 of the group
$\operatorname{OC}(\varGamma)$ for free $S$-modules $\varGamma$ of rank $n$ that span
$\LR^n$.

\section{Decomposition of coincidence isometries into reflections}\label{sec4}

The proofs of this section are parallel to the
corresponding ones of~\cite{zou}. For convenience, we prefer to
present the details. We start with the following relative of the well known
Cartan-Dieudonn\'e theorem on decomposing elements of
$\operatorname{O}(n)$ into a product of at most $n$ reflections; compare~\cite{omeara}. 

\begin{theorem}\label{refl}
Let $\varGamma\subset\LR^n$
be a free $S$-module of rank $n$ that spans
$\LR^n$. If any reflection
defined by a non-zero element of $\varGamma$ is a coincidence isometry
of $\varGamma$, then, any coincidence isometry of $\varGamma$ can be
decomposed as a product of at most $n$ reflections defined by non-zero elements of $\varGamma$.
\end{theorem}
\begin{proof}
Let $\{\gamma_1,\dots,\gamma_n\}$ be an $S$-basis of
$\varGamma$ and let $f\in\operatorname{OC}(\varGamma)$. We
argue by induction on $n$. Clearly, the assertion holds for
$n=1$. Assume that it holds for $n\geq 1$ and consider the case
$n+1$. 

Assume first that $f(\gamma_1)=\gamma_1$ and let $H$ be the hyperplane
orthogonal to $\gamma_1$, i.e., 
$$H:=\{x\in\LR^{n+1}\,|\,\langle x,\gamma_1\rangle=0\}\,.$$
Then, $H$ is an $n$-dimensional subspace of $\LR^{n+1}$ and, due to $f\in\operatorname{O}(n+1)$, $H$ is
$f$-invariant. More precisely, one has $f(H)=H$. It
follows that, via restriction, $f$ induces an isometry of the $n$-dimensional
Euclidean space $H$. Compare the orthogonal projection
$\pi\!:\,\LR^{n+1}\rightarrow H$ along $\gamma_1$, given by 
$$
\pi(x)=x-\frac{\langle x,\gamma_1\rangle}{\langle\gamma_1,\gamma_1\rangle}\gamma_1\,,
$$
with the reflection $\rho_{\gamma_1}$ of $\LR^{n+1}$ along $\gamma_1$,
given by
$$
\rho_{\gamma_1}(x)=x-2\frac{\langle x,\gamma_1\rangle}{\langle\gamma_1,\gamma_1\rangle}\gamma_1\,.
$$
By assumption, one has
$m:=[\rho_{\gamma_1}(\varGamma):(\varGamma\cap\rho_{\gamma_1}(\varGamma))]<\infty$,
wherefore $m\rho_{\gamma_1}(\gamma_i)\in\varGamma$ for all $1\leq i\leq
n+1$. Hence, 
$
2m\pi(\gamma_i)=m\gamma_i+m\rho_{\gamma_1}(\gamma_i)\in\varGamma
$
for all $1\leq i\leq
n+1$. In other words, \begin{equation}\label{eq}2m\pi(\varGamma)\subset\varGamma\,.\end{equation} Clearly,
$\pi(\varGamma)$ is a free $S$-module of rank $n$ with $S$-basis
$\{\pi(\gamma_2),\dots,\pi(\gamma_{n+1})\}$ and, further, $\pi(\varGamma)$ spans
the $n$-dimensional Euclidean space 
$H$. Secondly, we claim that the (co-)restriction $f|_H^H$ is a
coincidence isometry of $\pi(\varGamma)$. In order to see this, note
that, since
$f\in\operatorname{OC}(\varGamma)$, the subgroup index
$p:=[\varGamma:(\varGamma\cap f(\varGamma))]$ is finite, whence
$p\varGamma\subset (\varGamma\cap f(\varGamma))$. Further, since
$f|_H^H\pi=\pi f$, one obtains that
\begin{eqnarray*}
p\pi(\varGamma)\,\,\subset\,\,\pi(\varGamma\cap f(\varGamma))&\subset& \pi(\varGamma)\cap
f(\pi(\varGamma))\\&=&\pi(\varGamma)\cap f|_H^H(\pi(\varGamma))\,,
\end{eqnarray*}
thereby proving the claim by virtue of Theorem~\ref{charac}. Finally, any reflection of $\pi(\varGamma)$
defined by a non-zero element of
$\pi(\varGamma)$ is a coincidence isometry of $\pi(\varGamma)$. To see
this, consider the reflection $\rho_{\pi(\gamma)}|_H^H$ of $H$ along a
non-zero element $\pi(\gamma)$ of $\pi(\varGamma)$. By
Relation~\eqref{eq}, one has
$\rho_{\pi(\gamma)}=\rho_{2m\pi(\gamma)}=\rho_{\lambda}$, where
$\lambda:=2m\pi(\gamma)$ is a non-zero element of $\varGamma$. Since
$\rho_{\lambda}\in\operatorname{OC}(\varGamma)$ by assumption, the subgroup index
$q:=[\varGamma:(\varGamma\cap \rho_{\lambda}(\varGamma))]$ is finite, whence
$q\varGamma\subset (\varGamma\cap
\rho_{\lambda}(\varGamma))$. Further, since $\rho_{\lambda}|_H^H\pi=\pi \rho_{\lambda}$, one
obtains that
\begin{eqnarray*}
q\pi(\varGamma)\,\,\subset\,\,\pi(\varGamma\cap \rho_{\lambda}(\varGamma))&\subset& \pi(\varGamma)\cap
\rho_{\lambda}(\pi(\varGamma))\\&=&\pi(\varGamma)\cap \rho_{\lambda}|_H^H(\pi(\varGamma))\,,
\end{eqnarray*}
thereby proving the claim by virtue of Theorem~\ref{charac}. By
the induction hypothesis, $f|_H^H$ is a product of at most $n$
reflections defined by non-zero elements
of $\pi(\varGamma)$, say
$$f|_H^H=\rho_{\pi(\lambda_1)}|_H^H\dots\rho_{\pi(\lambda_j)}|_H^H\,.$$
By
Relation~\eqref{eq} together with $\rho_{\pi(\lambda_i)}=\rho_{2m\pi(\lambda_i)}$ for all $1\leq
i\leq j$, one obtains that
$$
f=\rho_{2m\pi(\lambda_1)}\dots\rho_{2m\pi(\lambda_j)}\,.
$$  
This completes the proof in this case.

Secondly, assume that $f(\gamma_1)\neq\gamma_1$, whence
$z:=f(\gamma_1)-\gamma_1\neq 0$. Since
$f\in\operatorname{OC}(\varGamma)$, the subgroup index
$m:=[f(\varGamma):(\varGamma\cap f(\varGamma))]$ is finite, wherefore $mz\in\varGamma$. Since $\rho_z=\rho_{mz}$, one sees that
$\rho(z)$
is a reflection defined by a non-zero element of $\varGamma$. It
follows that $\rho_z f\in\operatorname{OC}(\varGamma)$. One can
easily verify that $\rho_z f(\gamma_1)=\gamma_1$. Hence, by the
first case, $\rho_z f$ is the product of at most $n$ reflections
defined by non-zero vectors of $\varGamma$, say
$\rho_z f=\rho_{\lambda_1}\dots\rho_{\lambda_j}$. Since
$\rho_z^2=\operatorname{id}$, one has
$$f=\rho_z\rho_{\lambda_1}\dots\rho_{\lambda_j}=\rho_{mz}\rho_{\lambda_1}\dots\rho_{\lambda_j}$$
is a product of at most $n+1$ reflections defined by non-zero elements
of $\varGamma$. This completes
the proof in the second case.
\end{proof}

\begin{theorem}\label{bn}
Let $\varGamma\subset\LR^n$
be a free $S$-module of rank $n$ that spans
$\LR^n$ and let $\{\gamma_1,\dots,\gamma_n\}$ be an $S$-basis of
$\varGamma$. Then,
any reflection defined by a non-zero element of $\varGamma$ is a coincidence isometry
of $\varGamma$ if and only if, for all $1\leq i,j,k\leq n$, one has
$$
\frac{\langle \gamma_i,\gamma_j\rangle}{\langle \gamma_k,\gamma_k\rangle}\in K\,.
$$
\end{theorem}
\begin{proof}
Assume first that any reflection defined by a non-zero element of $\varGamma$ is a coincidence isometry
of $\varGamma$. In particular, the reflections $\rho_{\gamma_i}$,
where 
$1\leq i\leq n$, are coincidence isometries of $\varGamma$, whence
$m_i:=[\rho_{\gamma_i}(\varGamma):(\varGamma\cap\rho_{\gamma_i}(\varGamma))]<\infty$
and, further, $m_i\rho_{\gamma_i}(\varGamma)\subset\varGamma$ for all $1\leq i\leq n$. Since, for
all $1\leq i,j\leq n$, one has
$$
\rho_{\gamma_i}(\gamma_j)=\gamma_j-2\frac{\langle\gamma_j,\gamma_i\rangle}{\langle\gamma_i,\gamma_i\rangle}\gamma_i\,,
$$
the $\LR$-linear independence of the $\gamma_i$ shows that, for all $1\leq i,j\leq n$, one has
\begin{equation}\label{inK}
\frac{\langle \gamma_i,\gamma_j\rangle}{\langle \gamma_i,\gamma_i\rangle}\in K\,.
\end{equation}
If $\langle \gamma_i,\gamma_k\rangle\neq 0$, the last
observation shows that
$$
\frac{\langle \gamma_i,\gamma_i\rangle}{\langle \gamma_k,\gamma_k\rangle}=\frac{\langle \gamma_i,\gamma_i\rangle}{\langle \gamma_i,\gamma_k\rangle}\frac{\langle \gamma_i,\gamma_k\rangle}{\langle \gamma_k,\gamma_k\rangle}\in K\,.
$$
If $\langle \gamma_i,\gamma_k\rangle= 0$, consider the reflection
$\rho_{\lambda}$, where
$\lambda:=\gamma_i-\gamma_k\in\varGamma\setminus\{0\}$. By assumption,
$\rho_{\lambda}\in\operatorname{OC}(\varGamma)$ and similar to the
argumentation above, this gives
$$
\frac{\langle \gamma_i,\lambda\rangle}{\langle \lambda,\lambda\rangle}=\frac{1}{1+\frac{\langle \gamma_k,\gamma_k\rangle}{\langle \gamma_i,\gamma_i\rangle}}\in K\,,
$$
wherefore
$$
\frac{\langle \gamma_i,\gamma_i\rangle}{\langle \gamma_k,\gamma_k\rangle}\in K\,.
$$
Employing Equation~\eqref{inK}, this shows that, for all for all $1\leq i,j,k\leq n$, one has
$$
\frac{\langle \gamma_i,\gamma_j\rangle}{\langle \gamma_k,\gamma_k\rangle}=\frac{\langle \gamma_i,\gamma_j\rangle}{\langle \gamma_i,\gamma_i\rangle}\frac{\langle \gamma_i,\gamma_i\rangle}{\langle \gamma_k,\gamma_k\rangle}\in K\,.
$$
For the other direction, let $\gamma\in\varGamma\setminus\{0\}$, say
$\gamma=\sum_{j=1}^{n}\alpha_j\gamma_j$ for uniquely determined
$\alpha_j\in S$, $1\leq j\leq n$. Then, by assumption, for any $1\leq
i\leq n$, 
$$
\frac{\langle \gamma_i,\gamma\rangle}{\langle \gamma,\gamma\rangle}=\frac{\sum_{j=1}^{n}\alpha_j
  \langle \gamma_i,\gamma_j\rangle}{\sum_{k,l=1}^{n}\alpha_k\alpha_l \langle \gamma_k,\gamma_l\rangle}=\frac{\sum_{j=1}^{n}\alpha_j
  \frac{\langle \gamma_i,\gamma_j\rangle}{\langle \gamma_i,\gamma_i\rangle}}{\sum_{k,l=1}^{n}\alpha_k\alpha_l \frac{\langle \gamma_k,\gamma_l\rangle}{\langle \gamma_i,\gamma_i\rangle}}
$$
is an element of $K$. The reflection $\rho_{\gamma}$ of
$\LR^{n}$ along $\gamma$ satisfies 
$$
\rho_{\gamma}(x)=x-2\frac{\langle x,\gamma\rangle}{\langle \gamma,\gamma\rangle}\gamma\,.
$$
Since every element of $K$ is of the form $\alpha/\beta$, where
$\alpha\in S$ and $\beta\in S\setminus\{0\}$, one
sees that, for any $1\leq i\leq n$, there is an element
$\alpha_i\in S\setminus\{0\}$ such that
$\alpha_i\rho_{\gamma}(\gamma_i)\in\varGamma$. Setting
$\alpha:=\prod_{i=1}^n\alpha_i$, one obtains
$\alpha\rho_{\gamma}(\varGamma)\subset
(\varGamma\cap\rho_{\gamma}(\varGamma))$, 
which implies that $\rho_{\gamma}\in\operatorname{OC}(\varGamma)$ by
virtue of Theorem~\ref{charac}. This
completes the proof.
\end{proof}

The following consequence of Theorems~\ref{refl} and~\ref{bn} is immediate.

\begin{cor}\label{ci0}
Let $\varGamma\subset\LR^n$
be a free $S$-module of rank $n$ that spans
$\LR^n$. If there is a basis matrix $B_{\varGamma}\in
\operatorname{GL}(n,\LR)$ of $\varGamma$ such that
$B_{\varGamma}^tB_{\varGamma}$ has only entries in $K$, then every
non-zero vector of $\varGamma$ defines a coincidence reflection of $\varGamma$ and
every 
coincidence isometry of $\varGamma$ can be decomposed into a product of at most
$n$ reflections defined by non-zero elements of $\varGamma$.\qed
\end{cor}

\begin{definition}\label{omodule}
We call a
subset $\varGamma$ of $\LR^n$ an \emph{$S$-module over $K$ in $\LR^n$} if $\varGamma$ is a
free $S$-module of rank $n$ that spans $\LR^n$ and satisfies $\langle \gamma,\gamma\rangle\in
K$ for all $\gamma\in\varGamma$.
\end{definition}

Due to the polarization identity, the relation $\langle \gamma,\gamma\rangle\in
K$ holds for all $\gamma\in\varGamma$ if and only if one has $\langle \gamma,\gamma' \rangle\in K$ for all $\gamma,\gamma'\in \varGamma$. As an immediate consequence of Corollary~\ref{ci0}, one obtains the following

\begin{cor}\label{ci}
Let $\varGamma\subset\LR^n$ be an
$S$-module over $K$ in $\LR^n$. Then, every
non-zero vector of $\varGamma$ defines a coincidence reflection of $\varGamma$ and
every 
coincidence isometry of $\varGamma$ can be decomposed into a product of at most
$n$ reflections defined by non-zero elements of $\varGamma$. \qed  
\end{cor}

\begin{example}\label{ex}
The $\LZ$-modules over $\LQ$ in $\LR^n$ are
precisely the \emph{rational lattices} in $\LR^n$; cf.~\cite{crs,cs} for
specific examples. Interesting examples of $S$-modules over $K$ in $\LR^4$ are given by the 
{\em icosian ring} and the {\em octahedral}
 ring. In three dimensions, there are the body and face centred  
{\em icosahedral modules} of quasicrystallography; cf.~\cite{baake,bpr} and
references therein for details. An important class of planar examples is formed by the rings of {\em cyclotomic integers} in complex
cyclotomic fields.; cf.~\cite{PBR,Wa} and also see~\cite{glied}. These
rings appear in the description of planar mathematical
quasicrystals with $n$-fold cyclic symmetry; see~\cite{St} for genuine 
quasicrystals with cyclic symmetries of orders $5,8,10$ and $12$, respectively.
\end{example}

\section{Outlook}
In addition to the structure of the group of coincidence isometries, one is naturally interested in the characterization of the occuring coincidence submodules
and subgroup indices, respectively. We hope
 to report on progress in this direction for the above setting in the near future.

\section*{Acknowledgements}
It is a pleasure to thank Michael Baake, Svenja Glied and Uwe Grimm for helpful discussions.


\begin{thebibliography}{99}


\bibitem{baake} Baake, M.: Solution of the coincidence problem in
     dimensions $d\leq 4$. In: R. V. Moody (Ed.): The Mathematics
   of Long-Range Aperiodic Order. NATO-ASI Series C {\bf 489}, Kluwer,
 Dordrecht (1997), pp. 9--44; revised version \texttt{arXiv:math/0605222v1 [math.MG]}


\bibitem{bghz} Baake, M.; Grimm, U.; Heuer, M.; Zeiner, P.: Coincidence rotations of the root lattice $A_4$. European
  J. Combin., in press; \texttt{arXiv:0709.1341v1 [math.MG]}


\bibitem{bpr} Baake, M.; Pleasants, P. A. B.; Rehmann, U.: Coincidence site modules in $3$-space. Discrete Comput. Geom.  \textbf{38} (2007) 111--138; \texttt{arXiv:math/0609793v1 [math.MG]}.

\bibitem{bo} Borevich, Z. I.; Shafarevich, I. R.: Number Theory. Academic Press, New York (1966).


\bibitem{crs} Conway, J. H.; Rains, E. M.; Sloane, N. J. A.: On the existence of similar sublattices. Can. J. Math. \textbf{51} (1999) 1300--1306.
 
\bibitem{cs} Conway, J. H.; Sloane, N. J. A.: Sphere packings,
  Lattices and Groups. 3rd ed., Springer, New York (1999).

\bibitem{glied} Glied, S.: Similarity and coincidence isometries for modules. Submitted.

\bibitem{grimmer} Grimmer, H.: Coincidence-site lattices. Acta
  Cryst. A\textbf{32} (1976) 783--785


\bibitem{La} Lang, S.: Algebra. 3rd ed., Addison-Wesley, Reading, MA (1993).




\bibitem{N} Neukirch, J.: Algebraic Number Theory. Springer, Berlin (1999). 

\bibitem{omeara} O'Meara, O. T.: Inroduction to Quadratic Forms.
 3rd corr. printing, Springer, Berlin (1973).

\bibitem{PBR} Pleasants, P. A. B.; Baake, M.; Roth, J.: Planar
coincidences for $N$-fold symmetry. J. Math. Phys. \textbf{37}
(1996) 1029--1058; corr. version, \texttt{arXiv:math/0511147v1 [math.MG]}.  

\bibitem{St} Steurer, W.: Twenty years of structure research on
  quasicrystals. Part I. Pentagonal, octagonal, decagonal and
  dodecagonal quasicrystals. Z. Kristallogr. {\bf 219} (2004)
  391--446.
 
\bibitem{Wa} Washington, L. C.: Introduction to Cyclotomic Fields. 2nd ed., Springer, New York (1997).

\bibitem{zou} Zou, Y. M.: Structures of coincidence symmetry groups.
 Acta Cryst. A\textbf{62} (2006) 109--114.


\end{thebibliography}
\end{document}